\documentclass [12pt]{amsart}

\usepackage[left=1in,right=1in,top=1in,bottom=1in]{geometry} 

\usepackage{amsmath}

\usepackage{url}

\usepackage[normalem]{ulem}

\usepackage{amssymb,amsthm,amsmath}

\usepackage{ifthen}

\usepackage{graphicx}

\DeclareGraphicsExtensions{.eps}

\newtheorem{theorem}{Theorem}[section]

\newtheorem{lemma}[theorem]{Lemma}

\newtheorem{proposition}[theorem]{Proposition}

\newtheorem{corollary}[theorem]{Corollary}

\theoremstyle{definition}
\newtheorem{notation}[theorem]{Notation}
\newtheorem{definition}[theorem]{Definition}

\newtheorem{remark}[theorem]{Remark}

\numberwithin{equation}{section}

\newcommand{\C}{\mathbb{C}}

\newcommand{\e}{\epsilon}

\begin{document}
	\baselineskip=21pt
	\title[On stability in the B\"org's theorem for non-normal periodic Jacobi operators]{Stability in Non-Normal Periodic Jacobi Operators: Advancing B\"org's Theorem}
	\author{Krishna Kumar G.}
	\address{Department of Mathematics, University of Kerala, Kariavattom campus, Thiruvananthapuram, Kerala, India, 695581.}
	{\email{krishna.math@keralauniversity.ac.in}
	%	\thanks{Krishna Kumar G. (SQUID-1984-KG-1801) is supported by the SERB MATRICS grant with project reference no. MTR/2021/000028.}
		\author{V. B. Kiran Kumar}
		\address{Department of Mathematics, Cochin University of Science and Technology, Kerala, India.}
		{\email{vbk@cusat.ac.in}
		%	\thanks{V. B. Kiran Kumar is supported by KSCSTE, Kerala via KSYSA-Research Grant.}
			
			%    General info
			\subjclass[2020]{Primary 47B36; Secondary 47B37}
			
			%\date{January 1, 2001 and, in revised form, June 22, 2001.}
			
			
			\keywords{B\"org’s theorem, non-normal, periodic Jacobi operator, spectrum, pseudospectrum, path-connectedness.}

		\begin{abstract}
		Periodic Jacobi operators naturally arise in numerous applications, forming a cornerstone in various fields. The spectral theory associated with these operators boasts an extensive body of literature. Considered as discretized counterparts of Schr\"odinger operators, widely employed in quantum mechanics, Jacobi operators play a crucial role in mathematical formulations.
		
		The classical uniqueness result by G. B\"org in $1946$ occupies a significant place in the literature of inverse spectral theory and its applications. This result is closely intertwined with M. Kac's renowned article, 'Can one hear the shape of a drum?' published in $1966$. Since $1975,$ discrete versions of B\"org's theorem have been available in the literature.
		
		In this article, we concentrate on the non-normal periodic Jacobi operator and the discrete versions of B\"org's Theorem. We extend recently obtained stability results to encompass non-normal cases. The existing stability findings establish a correlation between the oscillations of the matrix entries and the size of the spectral gap.
		
		Our result encompasses the current self-adjoint versions of B\"org's theorem, including recent quantitative variations. Here, the oscillations of the matrix entries are linked to the path-connectedness of the pseudospectrum. Additionally, we explore finite difference approximations of various linear differential equations as specific applications.
		\end{abstract}

		\maketitle
		\textbf{Keywords:} Jacobi Operator; B\"org's theorem; Block Toeplitz matrices
		
		\section{Introduction}\label{sec1}
		
		The classical Sturm-Liouville problem involves determining the spectrum for a differential equation $y^{''}+(\lambda+q(x))y=0$ given the boundary conditions. In $1946,$ G. B\"org \cite{borg} addressed the inverse problem of recovering the potential function $q$ when the spectrum and associated boundary conditions are known. His work concludes that in the case of a periodic potential function $q$, the absence of spectral gaps implies that $q$ is constant almost everywhere. Additional related and extended results can be found in contemporary literature \cite{gelfand, krein-1, levinson-1, marchenko-0, marchenko-1, marchenko-2}. Complex potentials were explored in \cite{Buterin19, Hald, karaseva-1, Kwang04}, leading to a non-self-adjoint B\"org-type result.  The classical uniqueness result by G. B\"org, celebrated in inverse spectral theory and applications, is intricately connected to M. Kac's well-known article, "Can one hear the shape of a drum?" in $1966$ \cite{shape}. 
		
	The significance of inverse problems in mathematical physics has been comprehensively elucidated in various articles (see, for example, \cite{Bellman-1}). B\"org's techniques, initially integral to this domain, were later supplanted by more accessible tools, as noted by many mathematicians (see, for example, \cite{hoschd-1, hoschd-2}). The exploration of various adaptations of B\"org's results has remained an active area of research. For instance, in \cite{Fla, goldberg-1, goldberg-2}, authors have characterized potential functions in the presence of exactly $n$ spectral gaps. In cases where exactly one spectral gap exists, $q$ is identified as an elliptic function \cite{hoschd-1}.
	
Additionally, in \cite{Rya}, an estimation of the norm difference between two potentials, $q_1$ and $q_2$, is provided based on the differences in eigenvalues resulting from appropriate boundary problems. The findings in \cite{elec} carry direct applications to the inverse conductivity problem and lead to an analogue of the B\"org-Levinson theorem in multidimensions. Furthermore, \cite{Bellassoued20, Kian18, Pohjola18} presents a multidimensional B\"org-Levinson theorem for magnetic Schr\"odinger operators, while the vector-valued case is explored in \cite{brune95}. Generalizations and variations of these results can be found in \cite{Horv01, Huang, Yurko}. Expanding beyond the periodicity of the potential, a substantial body of literature exists for determining potential functions from scattering data \cite{Barry2000-1, Barry2000-2}. Similar inverse problems are addressed for matrix-valued operators in \cite{Carlson, Clark2002, clark2000}. An analogue for CMV operators with matrix-valued Verblunsky coefficients is obtained in \cite{Maxim10}, and a result for elliptic operators of higher order with constant coefficients is presented in \cite{Krupchyk12}.

A recent and intriguing area of research in spectral theory involves Schr\"odinger operators on finite trees.   B\"org-type problem for operators on finite trees is explored in \cite{Brown2005}. Various discrete versions of these problems are available in multiple articles \cite{Derevyagin06, Derevyagin03, kiran, Gol18, kikri, Shieh04}. A B\"org-type uniqueness theorem for a special class of unitary doubly infinite five-diagonal matrices is derived in \cite{Gesztesy06}, and uniqueness theorems for periodic Jacobi operators with matrix-valued coefficients are presented in \cite{Korotyaev09}.
	 
	 \subsection{From Continuous to Discrete}
	 
Jacobi operators, acting as discretized counterparts to Schr\"odinger operators widely employed in quantum mechanics, have a substantial presence in spectral theory \cite{Sodin}. In this brief overview, we outline the transition from the continuous to the discrete case, leading to the formulation of the Jacobi operator. Detailed explanations can be found in various articles \cite{kiran, kikri}.

It is noteworthy that the continuous case involves the spectrum of the Schr$\ddot{\mbox{o}}$dinger operator $\tilde{A}$, defined on a suitable subspace of $L^2(\mathbb{R})$ by:
\begin{equation}\label{Schrod}
	\tilde{A}u = {-u}^{''} + v \cdot u,
\end{equation}
The application of a straightforward finite difference approximation transforms the Schr$\ddot{\mbox{o}}$dinger operator defined in \eqref{Schrod} into a bounded self-adjoint operator on $\ell^2(\mathbb{Z})$, referred to as the discrete Schr$\ddot{\mbox{o}}$dinger operator. The pivotal concept, as introduced in \cite{kiran}, involves recognizing the discrete Schr$\ddot{\mbox{o}}$dinger operator as a doubly infinite block Toeplitz operator (following appropriate scaling and translation). This identification enables us to leverage the rich theory of spectral analysis applicable to block Toeplitz operators in addressing the problem. Here, we provide a brief explanation of this identification.
Consider the Schr$\ddot{\mbox{o}}$dinger operator with a periodic potential $v$, assuming without loss of generality that the periodicity of $v$ is $1$. Now, approximate the equation \eqref{Schrod} in the interval $[-n,n]$ with $n\in \mathbb{N}\cup \{\infty\}$ by using $p$ equispaced points in each interval $[j,j+1]\subset [-n,n]$. Employing the standard difference, we have:
\[
 {-u}^{''} (x_{i, (j)})= \frac{-u(x_{i+1,(j)})+2 u(x_{i,(j)}) - u(x_{i-1,(j)})}{h^2},
\]
with $h= 1/p$. For $j=-n,\ldots,n-1$, we have $\displaystyle{x_{s,(j)}=j+sh}$ where $s=0, \ldots, p-1$. Letting $n=\infty$, this can be treated as an operator denoted by $A$ acting on the sequence space $\ell^2(\mathbb{Z})$ and is defined by:
\begin{equation}\nonumber
	A(\{u_n\}_{n\in \mathbb{Z}})=
	\frac{-(u_{n-1}+u_{n+1})+2 u_n}{h^2}+v_nu_n; \,\,\{u_n\}_{n\in \mathbb{Z}}\in \ell^2(\mathbb{Z}),
\end{equation}
where the sequence $\{v_n\}_{n\in \mathbb{Z}}$ is obtained as the values of the periodic function $v$ at $p$ equispaced points in an interval of length $1$. The periodicity of $v$ implies that the sequence $\{v_n\}_{n\in \mathbb{Z}}$ is also periodic with period $p$. The matrix representation of $A$ with respect to the standard basis of $\ell^2(\mathbb{Z})$ is obtained as an infinite tridiagonal matrix (up to the scaling factor $h^2$):
\begin{equation}\nonumber
	A= \left[ \begin{array}{cccccccccc}
		\ddots & \ddots  \\
		\ddots & \ddots & \ddots & \\
		& -1 & 2+h^2v(x_{0, (1)}) & -1 & \\
		&   & \ddots &\ddots& \ddots \\
		&   &    &-1 & 2+h^2v(x_{p-1, (1)})  & -1\\
		&  &     &      & -1 & 2+h^2v(x_{0, (2)}) & -1 \\
		&  &     &      &    & \ddots & \ddots & \ddots \\
		&   &    &     &    &  & \ddots 	& \ddots
	\end{array} \right].
\end{equation}
Since $v$ is periodic, we have $v(x_{s,(j)})= v(x_{s,(j+1)})$ for all $s= 0, \ldots, p-1$. Thus, the sequence appearing in the main diagonal becomes periodic with period $p$. When $n$ is finite, the resulting matrix of size $np$ is the truncation of the bi-infinite matrix reported above. When $n=\infty,$ up to some scaling and translation by a scalar multiple of the identity operator, this operator can be identified as a doubly infinite block Toeplitz operator. In this article, we extend our considerations to a more general case involving periodic Jacobi operators.
	 \subsection{Periodic Jacobi operators}
Consider the doubly infinite matrix defined on $\ell^2({\mathbb{Z}})$,
\[
J:=\begin{bmatrix}
	\ddots& \ddots&\\
	\ddots& \ddots&\ddots\\
	&c_{-1}&a_0&b_0\\
	&&c_0&a_1&b_1\\
	&&&\ddots&\ddots&\ddots\\
	&&&&c_{p-2}&a_{p-1}&b_{p-1}\\
	&&&&&c_{p-1}&a_p&b_p\\
	&&&&&&\ddots&\ddots&\ddots \\
	&&&&&&&\ddots&\ddots&
\end{bmatrix}.
\]
Suppose $a=\{a_i\}_{i=-\infty}^\infty, b=\{b_i\}_{i=-\infty}^\infty$, and $c=\{c_i\}_{i=-\infty}^\infty$ are periodic sequences of real numbers with the same period. that is  $a_i= a_{p+i}$, $b_i= b_{p+i}$, and $c_i= c_{p+i}$ for every $i\in \mathbb{Z}$ and some $p\in \mathbb{N}$. Then $J$ is called a periodic Jacobi operator with period $p$. In general, $J$ is a non-normal bounded linear operator defined on $\ell^2(\mathbb{Z})$. Using periodicity, we can identify this operator  as the following doubly infinite block Toeplitz operator defined by
\[
J= \begin{bmatrix}
	\ddots&\ddots&\\
	\ddots&\ddots&\ddots\\
	&A_{1}&A_0&A_{-1}\\
	&&A_{1}&A_0&A_{-1}\\
	&&&\ddots&\ddots&\ddots\\
	&&&&\ddots&\ddots&
\end{bmatrix} \ \ \ \text{where} \ \ A_{-1}= \begin{bmatrix}
	&&&&&\\
	&&&&&\\
	&&&&&\\
	b_{p-1}&&&&
\end{bmatrix},
\]
\[
A_0= \begin{bmatrix}
	a_0&b_0\\
	c_0&a_1&b_1\\
	&c_1&a_2&b_2\\
	&&\ddots&\ddots&\ddots\\
	&&&c_{p-3}&a_{p-2}&b_{p-2}\\
	&&&&c_{p-2}&a_{p-1}
\end{bmatrix}, \ \ \ \text{and} \ \ \
A_{1}= \begin{bmatrix}
	&&&&c_{p-1}\\
	&&&&&\\
	&&&&&\\
	&&&&&\\
	&&&&&
\end{bmatrix}.
\] 
The symbol of $J$ is the matrix valued function $\phi$ defined by
\[
\phi(\theta)= \begin{bmatrix}
	a_0&b_0&&&&c_{p-1}e^{-i\theta}\\
	c_0&a_1&b_1\\
	&c_1&a_2&b_2\\
	&&\ddots&\ddots&\ddots\\
	&&&c_{p-3}&a_{p-2}&b_{p-2}\\
	b_{p-1}e^{i\theta}&&&&c_{p-2}&a_{p-1}
\end{bmatrix}, \ \ \ \  -\pi\leq \theta\leq \pi.
\]
If $a, b, c$ are constant sequences, then the periodic Jacobi operator $J$ becomes the doubly infinite Toeplitz operator with a symbol that is a scalar-valued function.

This article is dedicated to exploring the periodic Jacobi operator and establishing discrete versions of the celebrated B\"org's theorem. Our findings extend beyond normal cases, encompassing non-normal scenarios. The self-adjoint versions are comprehensively covered in \cite{Fla, kiran, KSN1}, along with spectral gap estimates detailed in \cite{Gol18, kikri}. To enhance clarity, we provide illustrative examples utilizing specific linear differential equations.

Furthermore, our objective is to derive a pseudospectral analogue of this result, specifically applicable to non-normal cases. It is important to note that pseudospectrum represents a neighborhood around the spectrum, primarily applied to normal matrices. Consequently, in the self-adjoint case, the path-connectedness of pseudospectrum is linked to the size of the spectral gap. Extending this concept, a non-normal analogue of our previous result establishes a relationship between the oscillations of the sequences $\{a_i\}_{i=0}^{p-1},         \{b_i\}_{i=0}^{p-1}$ with the path-connectedness of the pseudospectrum of the operator. In this pursuit, we extend the stability results obtained in \cite{Gol18} to non-normal cases.

	 \section{Preliminaries}
	 \subsection{Pseudospectra of operators}Consider a separable complex Hilbert space $H$. Let $BL(H)$ represent the algebra comprising all bounded linear operators on $H$, and $\sigma(A)$ denote the spectrum of a given operator $A \in BL(H)$. In the realm of normal operators, a wealth of information can be gleaned from their spectrum. However, when dealing with non-normal operators, the spectrum alone often falls short in providing norm-related insights about the operator. To bridge this gap, pseudospectrum emerges as a valuable tool, offering a nuanced understanding of the behavior exhibited by non-normal operators.
	\begin{definition}(\cite{tre})
		Let $A \in BL(H)$ and $\epsilon > 0$. The $\epsilon$-pseudospectrum of $A$ is defined as
		\[
		\Lambda_{\epsilon}(A) = \sigma(A) \cup \left\{z \in \mathbb{C}: \|(zI - A)^{-1}\| \geq \epsilon^{-1}\right\}.
		\]
		Equivalently, we have $
		\Lambda_{\epsilon}(A) = \bigcup\limits_{\|E\|\leq \epsilon}\{\sigma(A+E) \}.
	$
	\end{definition}
	
	Let $\mathbb{C}^{p\times p}$ denote the algebra of all $p \times p$ complex matrices. For $\epsilon > 0$, define
	\[
	\Delta_\epsilon = \{z \in \C: |z| \leq \epsilon\}.
	\]
	For $\Omega_1, \Omega_2 \subseteq \C$, define
	\[
	\Omega_1 + \Omega_2 = \{z_1 + z_2: z_1 \in \Omega_1, z_2 \in \Omega_2\}.
	\]
	
The pseudospectral theory is extended to operators on Banach spaces and to elements of unital Banach algebras; see \cite{aru, tre1}. Applications of pseudospectrum in various fields, particularly in scenarios involving non-normal matrices and operators, are thoroughly discussed in \cite{tre}. The subsequent properties of the pseudospectrum are later employed.
	
	 \begin{proposition}\label{prop1}
	 	Let $A, B \in BL(H)$ and $\e>0$. Then the following are true
	 	\begin{enumerate}
	 		\item \label{1} $\displaystyle \bigcap_{\e>0}\Lambda_\e(A)= \sigma(A)$.
	 		\item \label{2}  If $A$ is normal then $\Lambda_\e(A)= \sigma(A)+ \Delta_\e$.
	 		\item \label{3} $\Lambda_\e(A)+ \Delta_\delta\subseteq \Lambda_{\e+\delta}(A)$ for every $\delta>0$.
	 		\item  \label{4}  $\Lambda_{\e-\|B\|}(A)\subseteq \Lambda_\e(A+B)\subseteq \Lambda_{\e+\|B\|}(A)$ for every $\e\geq \|B\|$.
	 		\item \label{5}  If $\Lambda_\e(A)$ path-connected for some $\e>0$, then $\Lambda_\delta(A)$ is path-connected for every $\delta\geq \epsilon$.
	 
	 	\end{enumerate}
	 \end{proposition}
	 \begin{proof}
The identity \eqref{1} follows from the definition. For \eqref{2}, \eqref{3}, and \eqref{4}, one may refer to \cite{tre}.

Suppose $\Lambda_\epsilon(A)$ is path-connected for some $\epsilon > 0,$ and $\delta \geq \epsilon$. Then $\Lambda_\epsilon(A) \subseteq \Lambda_\delta(A)$, and every connected component of $\Lambda_\epsilon(A)$ and $\Lambda_\delta(A)$ has a nonempty intersection with $\sigma(A)$; see Theorem $4.3$ of \cite{tre}. If $\Lambda_\delta(A)$ is not path-connected, then it will be the disjoint union of path-connected components. One of them will contain $\Lambda_\epsilon(A)$, as it is path-connected, and others will be disjoint from it. However, all the components should intersect with $\sigma(A)\subset \Lambda_\epsilon(A)$. It follows that $\Lambda_\delta(A)$ is also path-connected.
	 \end{proof}

	 \subsection{From the literature on B\"org-type theorem}
	 We present a few recent results that can be considered as discrete B\"org-type theorems. These results shed light on the stability aspects, providing valuable insights into the   non-normal versions of the discrete B\"org theorem.
	\begin{theorem}\textnormal{(\cite{kri})}\label{spectincl}
		Let $J$ be the periodic Jacobi operator with a matrix-valued symbol $\phi$. Then,
		\begin{enumerate}
			\item $\sigma(J) = \bigcup_{-\pi \leq \theta \leq \pi} \sigma(\phi(\theta))$.
			\item $\Lambda_\epsilon(J) = \bigcup_{-\pi \leq \theta \leq \pi} \Lambda_\epsilon(\phi(\theta))$ for every $\epsilon > 0$.
		\end{enumerate}
	\end{theorem}
	
	 The following result is a straightforward consequence of Theorem \ref{spectincl}.
	 \begin{theorem}\label{mains1}
	 	Let $J$ be the periodic Jacobi operator with matrix valued symbol $\phi$ and $\e> 0$. If $a_i= a, b_i= b$ and $ c_i= c$ for every $i= 0, \ldots, p-1$, then $\sigma(J)$ and $\Lambda_{\epsilon}(J)$ are path-connected.
	 \end{theorem}
\begin{proof}
	From Theorem \ref{spectincl}, we have $\sigma(J) = \{a+be^{i\theta}+ce^{-i\theta}: -\pi\leq \theta\leq \pi\}$. Thus, $\sigma(J)$ is the continuous image of the path-connected set $[-\pi, \pi]$, making it path-connected. Also, for $\epsilon > 0$,
	\begin{align*}
		\Lambda_\epsilon(J) &= \{a+be^{i\theta}+ce^{-i\theta}+ \Delta_\epsilon: -\pi\leq \theta\leq \pi\} = \sigma(J) + \Delta_\epsilon.
	\end{align*}
	It follows that $\Lambda_\epsilon(J) $ is also path-connected.
\end{proof}
\begin{notation}\label{omega}
	For a sequence $v=(v_i)_{i=0}^{\infty}$, we define $\omega_v:= \max\{|v_i-v_j|: i,j=0,1,2,\ldots\}$ and $\omega_{v_0}:= \max\{|v_i-v_0|: i= 0,1,2,\ldots\}$. Notice that $\omega_{v_0}\leq \omega_v$. 
\end{notation}

	 The periodic Jacobi operator $J$ is self-adjoint if and only if $b_i=c_i$ for every $i= 0, 1, \ldots, p-1$. In that case, we have the following converse to the above result (see \cite{Gol18,kikri}).
	 \begin{theorem}\label{golkir}
	 	Let $J$ be a self-adjoint periodic Jacobi operator discussed above with period $p$. Then we have $\omega_b\leq (p-1)|\gamma|, \ \omega_a \leq p^2\sqrt{p} |\gamma|, \ \textrm{and} \	\omega_a +\omega_b\geq \frac{ |\gamma|}{4},$
	 	where $\gamma$ is the union of spectral gaps of $J.$
	 \end{theorem}
	 \begin{remark}
	 	The above result implies that the spectral gap size is in the order of the oscillation of the sequences $a,b.$ In particular, there are no gaps (that means $|\gamma|=0$) if and only if the sequences remain constant. This is precisely the statement of B\"org's theorem, and its discrete versions \cite{borg, Fla, kiran}.
	 \end{remark}
	\begin{remark}
		The discrete version of B\"org's theorem and its extensions to periodic Jacobi cases were proven in \cite{kiran} using the spectral inclusion results of block Toeplitz operators. These findings were further extended to establish the connection between the spectral gap size and the oscillation $\omega_{v}$ of the potential in \cite{kikri}. It is noteworthy that the results, as presented in Theorem \ref{golkir}, were independently obtained in \cite{Gol18}.
	\end{remark}
The primary contribution of this article lies in extending these results to non-normal cases. While B\"org-type results establish a connection between spectral gaps and the potential, our extension to the non-normal case focuses on the connection between the connectedness (or path-connectedness) of the pseudospectrum and the oscillation $\omega_{v}$ of the potential. The non-normal scenario introduces complexity, as the spectrum need not be real-valued. Since spectral information alone is insufficient to describe non-normal operators (due to the absence of a spectral theorem), we naturally consider pseudospectrum as a substitute for the spectrum.
	 \section{B\"org-type theorems for certain non-normal periodic Jacobi operators}
	In this section, we prove the B\"org-type theorem for non-normal periodic Jacobi operators of three kinds. Consider the non-normal periodic Jacobi operator with a matrix-valued symbol $\phi$ defined by
	 \[
	 \phi(\theta)= \begin{bmatrix}
	 	a_0&b_0&&&&c_{p-1}e^{-i\theta}\\
	 	c_0&a_1&b_1\\
	 	&c_1&a_2&b_2\\
	 	&&\ddots&\ddots&\ddots\\
	 	&&&c_{p-3}&a_{p-2}&b_{p-2}\\
	 	b_{p-1}e^{i\theta}&&&&c_{p-2}&a_{p-1}
	 \end{bmatrix}, \ \ \ \  -\pi\leq \theta\leq \pi.
	 \]
	This paper considers the following three cases.
	\begin{enumerate}
		\item $b_1= b_2= \cdots= b_{p-1}$ and $c_1= c_2= \cdots= c_{p-1}$.
		\item $b_1-c_1= b_2-c_2= \cdots= b_{p-1}-c_{p-1}$.
		\item $\{b_i\}$ and $ \{c_i\}$ are non-constant and periodic with the same period.
	\end{enumerate}
	In the above three cases, we obtain results that link the path-connectedness of the pseudospectra of the corresponding non-normal periodic Jacobi operator with the quantities $\omega_a, \omega_b,$ and $\omega_c$. We begin with the following technical lemma.
	 \begin{lemma}\label{theorem1}
	 	Let $\alpha, k\in \mathbb{C}$ and for $-\pi\leq \theta\leq \pi$, define $A(\theta), B(\theta)\in \C^{p\times p}$ by
	 	\[
	 	A(\theta)= \begin{bmatrix}
	 		0&\alpha&&&&\alpha e^{-i\theta}\\
	 		\alpha&0&\alpha\\
	 		&\alpha&0&\alpha\\
	 		&&\ddots&\ddots&\ddots\\
	 		&&&\alpha&{0}&\alpha\\
	 		\alpha e^{i\theta}&&&&\alpha&0
	 	\end{bmatrix}, \ \ \ B(\theta)= 
	 	\begin{bmatrix}
	 		&&&&&ke^{-i\theta}\\
	 		k&&\\
	 		&k&&\\
	 		&&\ddots&&\\
	 		&&&k&&\\
	 		&&&&k&
	 	\end{bmatrix}.
	 	\]
	 	Then the following are true.
	 	\begin{enumerate}
	 		\item $A(\theta)\,B(\theta)= B(\theta)\,A(\theta)$ for every $-\pi\leq \theta\leq \pi$.
	 		\item $B(\theta)\,B(\theta)^*= B(\theta)^*\,B(\theta)= |k|^2I$ for every $-\pi\leq \theta\leq \pi$.
	 		\item $A(\theta)+ B(\theta)$ is normal for every $-\pi\leq \theta\leq \pi$.
	 		\item $\sigma(B(\theta))= \left\{ke^{\frac{i(2r\pi-\theta)}{p}}: r= 0,1, \ldots, p-1\right\}.$
	 		\item $\sigma(A(\theta))= \left\{\alpha e^{\frac{i(2r\pi-\theta)}{p}}+ \alpha e^{\frac{-i(2r\pi-\theta)}{p}}: r= 0,1, \ldots, p-1\right\}.$
	 		\item $\sigma(A(\theta)+ B(\theta))= \left\{(\alpha+k)e^{\frac{i(2r\pi-\theta)}{p}}+ \alpha e^{\frac{-i(2r\pi-\theta)}{p}}: r= 0, 1, \ldots, p-1\right\}.$
	 	\end{enumerate}
	 \end{lemma}
	 \begin{proof}
	 	(1), (2), and (3) follows from the definition of $A(\theta)$ and $B(\theta)$.
	 	For the rest of the proof, define $\lambda_r= e^{\frac{i(2r\pi-\theta)}{p}},$ and $z_r= \begin{bmatrix}
	 		\lambda_r^{p-1} \\ \lambda_r^{p-2}\\ \vdots \\ 1
	 	\end{bmatrix}$,  for  $r=0,1,\ldots,p-1;$ and $-\pi\leq \theta\leq \pi$. Then $z_r$ is non-zero, and satisfies the identities;
 	
 	 $B(\theta)z_r= k\lambda_r z_r$, $A(\theta)z_r= \left(\alpha\lambda_r+ \frac{\alpha}{\lambda_r}\right)z_r$ and $	\left(A(\theta)+ B(\theta)\right) z_r= \left(\alpha\lambda_r+ \frac{\alpha}{\lambda_r}+ k\,\lambda_r\right)z_r.$ Hence we get 
 	
 \[
	 	\sigma(B(\theta))= \left\{ke^{\frac{i(2r\pi-\theta)}{p}}: r= 0,1, \ldots, p-1\right\},
	 	\] 
	 	\[
	 		\sigma(A(\theta))= \left\{\alpha e^{\frac{i(2r\pi-\theta)}{p}}+ \alpha e^{\frac{-i(2r\pi-\theta)}{p}}: r= 0,1, \ldots, p-1\right\}, \text{ and }
\] 
 \[
	 	\sigma(A(\theta)+ B(\theta))= \left\{(\alpha+k) e^{\frac{i(2r\pi-\theta)}{p}}+ \alpha e^{\frac{-i(2r\pi-\theta)}{p}}: r= 0,1, \ldots, p-1\right\}.
\] 
	 \end{proof}
	 \begin{corollary}\label{lem2}
	 	Let $-\pi\leq \theta\leq \pi$ and $A(\theta), B(\theta)$ be the matrix valued functions defined in Lemma \ref{theorem1}. Then the union $\bigcup\limits_{-\pi\leq\theta\leq \pi}\{\sigma(A(\theta)+ B(\theta))\}$ is path-connected.
	 \end{corollary}
	 \begin{proof}
	 	From $(6)$ of Lemma \ref{theorem1},
	 	\begin{align*}
	 	\bigcup\limits_{-\pi\leq\theta\leq \pi}	\sigma(A(\theta)+ B(\theta))&= \bigcup_{-\pi\leq\theta\leq \pi}\left\{(\alpha+k) e^{\frac{i(2r\pi-\theta)}{p}}+ \alpha e^{\frac{-i(2r\pi-\theta)}{p}}: r= 0,1, \ldots, p-1\right\}\\
	 		&= \bigcup_{r=0}^{p-1}\left\{(\alpha+k) e^{\frac{i(2r\pi-\theta)}{p}}+ \alpha e^{\frac{-i(2r\pi-\theta)}{p}}: -\pi\leq\theta\leq \pi\right\}\\
	 		&= \bigcup_{r=0}^{p-1}\left\{\alpha\left(e^{\frac{i(\theta-2r\pi)}{p}}+ e^{\frac{-i(\theta-2r\pi)}{p}}\right)+ k e^{\frac{i(2r\pi-\theta)}{p}}: -\pi\leq\theta\leq \pi\right\}\\
	 		&= \bigcup_{r=0}^{p-1}\left\{2\alpha\cos \left(\frac{\theta-2r\pi}{p}\right)+ k e^{\frac{i(2r\pi-\theta)}{p}}: -\pi\leq\theta\leq \pi\right\}.
	 	\end{align*}
	 	For $r= 0, 1, \ldots, p-1$, define $\Omega_r= \left\{2\alpha\cos \left(\frac{\theta-2r\pi}{p}\right)+ k e^{\frac{i(2r\pi-\theta)}{p}}: -\pi\leq\theta\leq \pi\right\}$. Then $\Omega_r$ is path-connected and $\Omega_r \cap \Omega_{r+1}\neq \emptyset$. Hence 
	 	\[
	 	\{\sigma(A(\theta)+ B(\theta)): -\pi\leq\theta\leq \pi\}= \bigcup_{r=0}^{p-1} \Omega_r
	 	\]
	 	is path-connected.
	 \end{proof}
	 \subsection{Special case $b_i's$ and $c_i's$ are constants}\label{case1}
	 \begin{theorem}\label{pseudo}
	 	Let $J$ be the periodic Jacobi operator with matrix valued symbol $\phi$. If $b_1=  \cdots= b_{p-1}$ and $c_1= \cdots= c_{p-1}$, then $\Lambda_{3\omega_{a_0}}(J)$ is path-connected.
	 \end{theorem}
	 \begin{proof}
	 Suppose $b_1 = \cdots = b_{p-1} = b$ and $c_1 = \cdots = c_{p-1} = c$. Then, $\phi(\theta) = f(\theta) + E$, $-\pi \leq \theta \leq \pi$, where
	 	\[
	 	f(\theta)= \begin{bmatrix}a_0&b&&&ce^{- i\theta}\\c&a_0&b&&&\\
	 		&\ddots&\ddots&\ddots&\\
	 		&&c&a_{0}&b\\
	 		be^{i\theta}&&&c&a_{0}\end{bmatrix} \ \ \text{and} \ \
	 	E= \begin{bmatrix}0&&\\
	 		&a_1-a_0&\\
	 		&&\ddots&\\
	 		&&&a_{p-2}-a_0&\\
	 		&&&&a_{p-1}-a_0\end{bmatrix}.
	 	\]
	 From (4) of Proposition \ref{prop1}, with $A = f(\theta)$, $B = E$, and $\epsilon \geq \|E\|$, we obtain
	 \begin{equation}\label{eqincl}
	 	\bigcup_{-\pi \leq \theta \leq \pi} \Lambda_{\epsilon - \|E\|}(f(\theta)) \subseteq \bigcup_{-\pi \leq \theta \leq \pi} \Lambda_{\epsilon}(f(\theta) + E) \subseteq \bigcup_{-\pi \leq \theta \leq \pi} \Lambda_{\epsilon + \|E\|}(f(\theta)).
	 \end{equation}
	 
	 Taking $\epsilon = \|E\|$ in \eqref{eqincl}, and utilizing the fact that $f(\theta)$ is normal (Lemma \ref{theorem1}), we obtain
	 \begin{equation}\label{eqincl1}
	 	\begin{aligned}
	 		\bigcup_{-\pi \leq \theta \leq \pi}\sigma(f(\theta)) \subseteq \bigcup_{-\pi \leq \theta \leq \pi} \Lambda_{\|E\|}(f(\theta)+E) \
	 		=  \Lambda_{\|E\|}(J) \subseteq 	\bigcup_{-\pi \leq \theta \leq \pi}\Lambda_{2\|E\|}(f(\theta)) =\bigcup_{-\pi \leq \theta \leq \pi}\sigma(f(\theta)) + \Delta_{2\|E\|}.
	 	\end{aligned}
	 \end{equation}
	 
From Corollary \ref{lem2}, it is evident that the leftmost and rightmost sets in the above inclusions are path-connected. Therefore, $\Lambda_{\|E\|}(J) + \Delta_{2\|E\|}$ is also path-connected. It is worth noting that the set $\Lambda_{\|E\|}(J)$ is sandwiched between $\bigcup_{-\pi \leq \theta \leq \pi}\sigma(f(\theta))$ and the union of discs of radii $2\|E\|$, centered in the first set. This arrangement ensures that the union of discs of radii $2\|E\|$, with centers in $\Lambda_{\|E\|}(J)$, remains path-connected.
	 
	 Next, we claim that $\Lambda_{3\|E\|}(J)$ is path-connected.
	 From (3) of Proposition \ref{prop1}, we have 
	 \begin{equation}\label{inclu1}
	\begin{aligned}
		\Lambda_{\|E\|}(J) + \Delta_{2\|E\|} \subseteq \Lambda_{3\|E\|}(J).
			\end{aligned}
	\end{equation}
 This inclusion, along with the inclusions in \eqref{eqincl1}, yields the following inclusions;
 \begin{equation}\label{inclu2}
	\begin{aligned}
	 	\bigcup_{-\pi\leq \theta \leq \pi} \sigma(f(\theta))+ \Delta_{2\|E\|}\subseteq  \Lambda_{\|E\|}(J)+\Delta_{2\|E\|}\subseteq \Lambda_{3\|E\|}(J).
	\end{aligned}
\end{equation}
 \begin{equation}\label{inclu3}
	\begin{aligned}
\sigma(J) \subseteq \Lambda_{\|E\|}(J) \subseteq \bigcup_{-\pi \leq \theta \leq \pi} \sigma(f(\theta)) + \Delta_{2\|E\|}.
 	\end{aligned}
\end{equation}
 \begin{equation}\label{inclu4}
	\begin{aligned}
			 \sigma(J) \subseteq \bigcup_{-\pi \leq \theta \leq \pi} \sigma(f(\theta)) + \Delta_{2\|E\|}\subseteq \Lambda_{3\|E\|}(J).
			  	\end{aligned}
		 \end{equation}
	Here the middle term $\bigcup_{-\pi \leq \theta \leq \pi} \sigma(f(\theta)) + \Delta_{2\|E\|}$ is path-connected, and
	 each path-connected component of $\Lambda_{3\|E\|}(J)$ will intersect with $\sigma(J)$.
Hence $\Lambda_{3\|E\|}(J)$ is path-connected. Since $\|E\| \leq \omega_{a_0}$ (from (5) of Proposition \ref{prop1}), $\Lambda_{3\omega_{a_0}}(J)$ is path-connected. This concludes the proof.
	 \end{proof}
	 The converse question is addressed in the following theorem.
	 \begin{theorem}\label{pseudoconverse}
	 	Let $J$ be the periodic Jacobi operator with matrix valued symbol $\phi$. Further $b_1= \cdots= b_{p-1}= b$ and $c_1= \cdots= c_{p-1}= c$. If $\Lambda_{\epsilon}(J)$ is path-connected for some $\e\geq |b-c|$, then $\omega_a\leq 2(\epsilon+|b-c|)(p-1).$
	 \end{theorem}
	 \begin{proof}
	 	Suppose $b_1= \cdots= b_{p-1}= b$ and $c_1= \cdots= c_{p-1}= c$. Then $\phi(\theta)= \phi_1(\theta)+ \phi_2(\theta)$, $-\pi\leq \theta\leq \pi$, where
	 	\[
	 	\phi_1(\theta)= \begin{bmatrix}a_0&b&&&be^{- i\theta}\\
	 		b&a_1&b&&&\\
	 		&\ddots&\ddots&\ddots&\\
	 		&&b&a_{p-2}&b\\
	 		be^{i\theta}&&&b&a_{p-1}\end{bmatrix}, \ \ 
	 	\phi_2(\theta)= \begin{bmatrix}&&&&(c-b)e^{- i\theta}\\ c-b&&&&\\
	 		&\ddots&&&\\
	 		&&&c-b&\end{bmatrix},
	 	\]
	 	and $\|\phi_2(\theta)\|^2= \sigma_{\max}(\phi_2(\theta)^*\phi_2(\theta))=|c-b|^2$. From (4) of Proposition \ref{prop1}, for $\e\geq |c-b|$, 
	 	\[
	 	\bigcup_{-\pi\leq \theta \leq \pi} \Lambda_{\epsilon-|c-b|}(\phi_1(\theta))\subseteq  \bigcup_{-\pi<\theta\leq \pi} \Lambda_{\epsilon}(\phi(\theta))\subseteq \bigcup_{-\pi\leq \theta\leq \pi} \Lambda_{\epsilon+|c-b|}(\phi_1(\theta)).
	 	\]
	 	Note that $\phi_1(\theta)$ is the matrix valued symbol function of the periodic self-adjoint discrete Schr\"odinger operator. Since $\phi_1(\theta)$ is normal, $\Lambda_{\epsilon-|c-b|}(\phi_1(\theta))=\sigma(\phi_1(\theta))+ \Delta_{\epsilon-|c-b|}$ and  $\Lambda_{\epsilon+|c-b|}(\phi_1(\theta))= \sigma(\phi_1(\theta))+ \Delta_{\epsilon+|c-b|}$. Thus
	 	\[
	 	\bigcup_{-\pi\leq \theta\leq \pi} \sigma(\phi_1(\theta))+ \Delta_{\epsilon- |c-b|}\subseteq  \Lambda_\epsilon(J)\subseteq \bigcup_{-\pi\leq \theta\leq \pi} \sigma(\phi_1(\theta))+ \Delta_{\epsilon+ |c-b|}.
	 	\]
	 	The left and right sets are unions of discs with centers $\left\{\sigma(\phi_1(\theta)): -\pi\leq \theta\leq \pi\right\}\subseteq \mathbb{R}$ and radii $\epsilon-|c-b|$ and $\epsilon+ |c-b|$ respectively. The set $\Lambda_\e(J)$ in between them is path-connected implies $\displaystyle \bigcup_{-\pi\leq \theta\leq \pi} \sigma(\phi_1(\theta))+ \Delta_{\epsilon+ |c-b|}$ path-connected. Thus the spectral gap size of the self-adjoint discrete Schr\"odinger operator with symbol $\phi_1(\theta)$ is at most $\epsilon+|c-b|$ (see \cite{kikri}). From Theorem 3.1 of \cite{kikri}, $\omega_a\leq 2(\epsilon+ |c-b|)(p-1).$
	 \end{proof}
\begin{remark}\label{pseudorem}
	Consider the case where $b = c$, and contemplate the self-adjoint periodic Jacobi operator $J$. In this scenario, $\Lambda_{\omega_{a_0}}(J) = \sigma(J) + \Delta_{\omega_{a_0}}$. When the diagonal sequence $a = \{a_i\}_{i=0}^{p-1}$ is constant, it follows that $\omega_{a_0} = 0$. According to Theorem \ref{pseudo}, the spectrum becomes path-connected. Conversely, if the spectrum is path-connected, then $\Lambda_0(J)$ also exhibits path-connectedness. By applying Theorem \ref{pseudoconverse}, it is concluded that $\omega_{a} = 0$, implying $a_0 = a_1 = \cdots = a_{p-1}$. Consequently, Theorem \ref{pseudo} and Theorem \ref{pseudoconverse} serve as the pseudospectral analogues of the discrete B\"org-type theorem for discrete Schr\"odinger operators with constant coefficients; refer to \cite{kiran, kikri}.
\end{remark}
	 
	 \subsection{Special case $c_i-b_i=k$ for every $i= 0, 1,\ldots, p-1$}\label{case2}
	 \begin{theorem}\label{pseudo1}
	 	Let $J$ be the periodic Jacobi operator with matrix valued symbol $\phi$. If $b_i-c_i= k$ for every $i= 0, 1,\ldots,p-1$, then $\Lambda_{3\left(\sqrt{(p-1)\left(\omega_{a_0}^2+2\omega_{b_0}^2\right)}\right)}(J)$ is path-connected.
	 \end{theorem}
	 \begin{proof}
	 	For $-\pi\leq \theta\leq \pi$, define $\phi(\theta)= a_0I+f_1(\theta)+ f_2(\theta)+ f_3(\theta)$ where
	 	\[
	 	f_1(\theta)=\begin{bmatrix}
	 		0&b_0&&&&b_{0}e^{-i\theta}\\
	 		b_0&0&b_0\\
	 		&b_0&0&b_0\\
	 		&&\ddots&\ddots&\ddots\\
	 		&&&b_{0}&{0}&b_{0}\\
	 		b_{0}e^{i\theta}&&&&b_{0}&0
	 	\end{bmatrix}, \ \
	 	f_2(\theta)=
	 	\begin{bmatrix}
	 		&&&&&ke^{-i\theta}\\
	 		k&&\\
	 		&k&&\\
	 		&&\ddots&&\\
	 		&&&k&&\\
	 		&&&&k&
	 	\end{bmatrix}, \ \ \text{and}
	 	\]
	 	\[
	 	f_3(\theta)= \begin{bmatrix}
	 		0&0&&&&(b_{p-1}-b_{0})e^{-i\theta}\\
	 		0&a_1-a_0&b_1-b_0\\
	 		&b_1-b_0&a_2-a_0&b_2-b_0\\
	 		&&\ddots&\ddots&\ddots\\
	 		&&&b_{p-3}-b_{0}&a_{p-2}-a_0&b_{p-2}-b_{0}\\
	 		(b_{p-1}-b_{0})e^{i\theta}&&&&b_{p-2}-b_{0}&a_{p-1}-a_0
	 	\end{bmatrix}.
	 	\]
	 	From (4) of Proposition \ref{prop1}, for $\e\geq \|f_3(\theta)\|$,
	 	\[
	 	\Lambda_{\e-\|f_3(\theta)\|}(a_0I+f_1(\theta)+ f_2(\theta))\subseteq \Lambda_{\e}(\phi(\theta)) \subseteq \Lambda_{\e+\|f_3(\theta)\|}(a_0I+f_1(\theta)+ f_2(\theta)).
	 	\]
	 	Denote $\phi'(\theta)= a_0I+f_1(\theta)+ f_2(\theta)$. From (3) of Lemma \ref{theorem1}, $\phi'(\theta)$ is normal, and
	 	\[
	 	\sigma(\phi'(\theta))+ \Delta_{\e-\|f_3(\theta)\|}\subseteq \Lambda_{\e}(\phi(\theta)) \subseteq \sigma(\phi'(\theta))+ \Delta_{\e+\|f_3(\theta)\|}.
	 	\]
	 	Thus
	 	\begin{align*}
	 		\bigcup_{-\pi\leq \theta\leq\pi}\sigma(\phi'(\theta))+ \Delta_{\e-\|f_3(\theta)\|} \subseteq  \bigcup_{-\pi\leq \theta\leq\pi} \Lambda_{\e}(\phi(\theta)) \subseteq \bigcup_{-\pi\leq \theta\leq\pi}\sigma(\phi'(\theta))+ \Delta_{\e+\|f_3(\theta)\|}.
	 	\end{align*}
	 	From (6) of Lemma \ref{theorem1},
	 	\begin{align*}
	 		\bigcup_{-\pi\leq \theta\leq\pi}\sigma(\phi'(\theta))&= \bigcup_{-\pi\leq\theta\leq \pi} a_0+ \left\{c_0e^{\frac{i(2r\pi-\theta)}{p}}+ b_0e^{\frac{-i(2r\pi-\theta)}{p}}: r= 0, 1, \ldots, p-1\right\}.
	 	\end{align*}
	 	By the similar arguments in the proof of Theorem \ref{pseudo}, we get $\displaystyle \bigcup_{-\pi\leq \theta\leq\pi} \Lambda_{3\|f_3(\theta)\|}(\phi(\theta))$ is path-connected. We also have $\|f_3(\theta)\|\leq\|f_3(\theta)\|_{F}\leq \sqrt{(p-1)\left(\omega_{a_0}^2+2\omega_{b_0}^2\right)}$ for every $-\pi\leq \theta\leq \pi$, where $\|f_3(\theta)\|_{F}$ denotes the Frobenius norm of $\|f_3(\theta)\|$. From (5) of Proposition \ref{prop1},
	 	\[
	 	\bigcup_{-\pi\leq \theta\leq\pi} \Lambda_{3\left(\sqrt{(p-1)\left(\omega_{a_0}^2+2\omega_{b_0}^2\right)}\right)}(\phi(\theta))= \Lambda_{3\left(\sqrt{(p-1)\left(\omega_{a_0}^2+2\omega_{b_0}^2\right)}\right)}(J),
	 	\]
	 	is path-connected.
	 \end{proof}
	 \begin{theorem}\label{pseudo1conv}
	 	Let $J$ be the periodic Jacobi operator with matrix valued symbol $\phi$ and $b_i-c_i= k$ for every $i= 0, 1,\ldots,p-1$. If $\Lambda_{\epsilon}(J)$ is path-connected for some $\e\geq |k|$, then $\omega_a\leq 2(\epsilon+|k|)(p-1)$.
	 \end{theorem}
	 \begin{proof}
	 	For $-\pi\leq \theta \leq \pi$, define $\phi(\theta)= \phi_1(\theta)+ \phi_2(\theta)$ where
	 	\[
	 	\phi_1(\theta)= \begin{bmatrix}a_0&b_0&0&\cdots&0&b_{p-1}e^{- i\theta}\\b_0&a_1&b_1&\cdots&0&0\\\vdots&\vdots&\cdots&\vdots&\vdots&\vdots\\0&0&0&\cdots&a_{p-2}&b_{p-2}\\
	 		b_{p-1}e^{i\theta}&0&0&\cdots&b_{p-2}&a_{p-1}\end{bmatrix} \ \ \text{and} \ \
	 	\phi_2(\theta)= \begin{bmatrix}0&0&\cdots&0&ke^{- i\theta}\\ k&0&\cdots&0&0\\\vdots&\vdots&\vdots&\vdots&\vdots\\0&0&\cdots&0&0\\
	 		0&0&\cdots&k&0\end{bmatrix}.
	 	\]
	 	From (2) of Lemma \ref{theorem1}, $\|\phi_2(\theta)\|= |k|$ and the result follows. The proof is a mere imitation of Theorem \ref{pseudoconverse}.
	 \end{proof}

\begin{remark}\label{pseudo1remark}
	The significance of Theorem \ref{pseudo1} lies in demonstrating that the path-connectedness of the pseudospectrum of a periodic Jacobi operator is proportional to the oscillations $\omega_{a_0}$ and $\omega_{b_0}$ with an apparent dependence on the period $p$. Consider the case of a self-adjoint periodic Jacobi operator $J$ (i.e., $b_i = c_i$ for every $i = 0,1,\ldots, p-1$). In this scenario, $\Lambda_{\sqrt{(p-1)\left(\omega_{a_0}^2+2\omega_{b_0}^2\right)}}(J) = \sigma(J) + \Delta_{\sqrt{(p-1)\left(\omega_{a_0}^2+2\omega_{b_0}^2\right)}}$. When the diagonal sequence $a=\{a_i\}_{i=0}^{p-1}$ is constant, it follows that $\omega_{a_0} = 0$. As per Theorem \ref{pseudo1}, $\sigma(J) + \Delta_{3\left(\sqrt{2(p-1)}\omega_{b_0}\right)}$ is path-connected. Conversely, if the spectrum is path-connected, then $\Lambda_0(J)$ also demonstrates path-connectedness. By utilizing Theorem \ref{pseudo1conv}, it is concluded that $\omega_{a} = 0$, and consequently, $a_0 = a_1 = \cdots = a_{p-1}$. Thus, Theorem \ref{pseudo1} and Theorem \ref{pseudo1conv} stand as pseudospectral analogues of the discrete B\"org-type theorem for discrete Schr\"odinger operators with periodic coefficients.
\end{remark}
	 
	 \subsection{Special case $a,b,c$ are periodic with same period $p$}\label{case3}
	 \begin{theorem}\label{pseudo2}
	 	Let $J$ be the periodic Jacobi operator with matrix valued symbol $\phi$ and $a= \{a_i\}_{i=0}^{p-1}, b= \{b_i\}_{i=0}^{p-1}$, $c= \{c_i\}_{i=0}^{p-1}$ are of same period. Define $\omega_{b_0c}= \max\{|b_0-c_i|: i= 0, 1, \ldots, p-1\}$, then $\Lambda_{3\left(\sqrt{(p-1)\left(\omega_{a_0}^2+\omega_{b_0}^2+\omega_{b_0c}^2\right)}\right)}(J)$ is path-connected.
	 \end{theorem}
	 \begin{proof}
	 	For $-\pi\leq \theta\leq \pi$, define $\phi(\theta)= a_0I+f_1(\theta)+f_2(\theta)$ where 
	 	\[
	 	f_1(\theta)=\begin{bmatrix}
	 		0&b_0&&&b_0e^{-i\theta}\\
	 		b_0&0&b_0\\
	 		&\ddots&\ddots&\ddots\\
	 		&&b_0&0&b_0\\
	 		b_0e^{i\theta}&&&b_0&0
	 	\end{bmatrix}\]
	 	and 
	 	\[
	 	f_2(\theta)=\begin{bmatrix}
	 		0&\cdots&&&0&(c_{p-1}-b_0)e^{-i\theta}\\
	 		c_0-b_0&a_1-a_0&b_1-b_0\\
	 		&c_1-b_0&a_2-a_0&b_2-b_0\\
	 		&&\ddots&\ddots&\ddots\\
	 		&&&c_{p-3}-b_0&a_{p-2}-a_0&b_{p-2}-b_0\\
	 		(b_{p-1}-b_0)e^{i\theta}&&&&c_{p-2}-b_{0}&a_{p-1}-a_0
	 	\end{bmatrix}.\]
	 	From (4) of Proposition \ref{prop1}, for $\e\geq \|f_2(\theta)\|$,
	 	\[
	 	\Lambda_{\e-\|f_2(\theta)\|}(a_0I+f_1(\theta))\subseteq \Lambda_{\e}(\phi(\theta)) \subseteq \Lambda_{\e+\|f_2(\theta)\|}(a_0I+f_1(\theta)).
	 	\]
	 	From (3) of Lemma \ref{theorem1}, $a_0I+f_1(\theta)$ is normal and
	 	\[
	 	\sigma(a_0I+f_1(\theta))+ \Delta_{\e-\|f_2(\theta)\|}\subseteq \Lambda_{\e}(\phi(\theta))\subseteq \sigma(a_0I+f_1(\theta))+ \Delta_{\e+\|f_2(\theta)\|}.
	 	\]
	 	Thus
	 	\begin{align*}
	 		\bigcup_{-\pi\leq \theta\leq \pi}\sigma(\phi(\theta))+ \Delta_{\e+\|f_2(\theta)\|} \subseteq  \bigcup_{-\pi\leq \theta\leq \pi} \Lambda_{\e}(\phi(\theta)) \subseteq \bigcup_{-\pi\leq \theta\leq\pi}\sigma(\phi(\theta))+ \Delta_{\e+\|f_2(\theta)\|}.
	 	\end{align*}
	 	From (6) of Lemma \ref{theorem1},
	 	\begin{align*}
	 		\bigcup_{-\pi\leq \theta<\pi}\sigma(\phi(\theta))&= \bigcup_{-\pi<\theta\leq \pi} a_0+ \left\{b_0e^{\frac{-i(2r\pi-\theta)}{p}}: r= 0, 1, \ldots, p-1\right\}.
	 	\end{align*}
	 	By the similar arguments in the proof of Theorem \ref{pseudo}, we get $\displaystyle \bigcup_{-\pi\leq \theta\leq \pi} \Lambda_{3\|f_2(\theta)\|}(\phi(\theta))$ is path-connected. We also have $\|f_2(\theta)\|\leq\|f_2(\theta)\|_{F}\leq \sqrt{(p-1)\left(\omega_{a_0}^2+\omega_{b_0}^2+\omega_{b_0c}^2\right)}$ for every $-\pi\leq \theta\leq \pi$. From (5) of Proposition \ref{prop1}, 
	 	\[
	 	\bigcup_{-\pi\leq \theta\leq \pi} \Lambda_{3\left(\sqrt{(p-1)\left(\omega_{a_0}^2+\omega_{b_0}^2+\omega_{b_0c}^2\right)}\right)}(\phi(\theta))= \Lambda_{3\left(\sqrt{(p-1)\left(\omega_{a_0}^2+\omega_{b_0}^2+\omega_{b_0c}^2\right)}\right)}(J)
	 	\]
	 	is path-connected. 
	 \end{proof}
	 \begin{theorem}\label{pseudo2conv}
	 	Let $J$ be the periodic Jacobi operator with matrix valued symbol $\phi$ and $a= \{a_i\}_{i=0}^{p-1}, b= \{b_i\}_{i=0}^{p-1}$, $c= \{c_i\}_{i=0}^{p-1}$ are of same period. Define $\omega_{bc}= \max\{|b_i-c_i|: i= 0, 1, \ldots, p-1\}$. If $\Lambda_{\epsilon}(J)$ is path-connected for some $\e\geq \omega_{bc}$, then $\omega_a\leq 2(\epsilon+\omega_{bc})(p-1)$.
	 \end{theorem}
	 \begin{proof}
	 	For $-\pi\leq \theta\leq \pi$, define $\phi(\theta)= \phi_1(\theta)+ \phi_2(\theta)$ where
	 	\[
	 	\phi_1(\theta)= 
	 	\begin{bmatrix}a_0&b_0&0&\cdots&0&b_{p-1}e^{- i\theta}\\b_0&a_1&b_1&\cdots&0&0\\\vdots&\vdots&\cdots&\vdots&\vdots&\vdots\\0&0&0&\cdots&a_{p-2}&b_{p-2}\\
	 		b_{p-1}e^{i\theta}&0&0&\cdots&b_{p-2}&a_{p-1}
	 	\end{bmatrix}
	 	\]
	 	and
	 	\[
	 	\phi_2(\theta)= \begin{bmatrix}0&0&\cdots&0&(c_{p-1}-b_{p-1})e^{- i\theta}\\ c_{0}-b_{0}&0&\cdots&0&0\\\vdots&\vdots&\vdots&\vdots&\vdots\\0&0&\cdots&0&0\\
	 		0&0&\cdots&c_{p-2}-b_{p-2}&0\end{bmatrix}.
	 	\]
	 	From (2) of Lemma \ref{theorem1}, $\|\phi_2(\theta)\|= \omega_{bc}$ and the result follows. The proof is a mere imitation of Theorem \ref{pseudoconverse}.
	 \end{proof}
\begin{remark}\label{pseudo2remark}
	The interconnected results of Theorem \ref{pseudo2} and Theorem \ref{pseudo2conv} highlight that the path-connectedness of the pseudospectrum of a periodic Jacobi operator is aligned with the order of oscillations $\omega_{a_0}$, $\omega_{b_0}$, $\omega_{b_0c}$, and $\omega_{bc}$ with an apparent dependency on the period $p$. These theorems share a common theme with Theorem \ref{pseudo} (corresponding to Theorem \ref{pseudo1}) and Theorem \ref{pseudo2conv} aligns with the spirit of both Theorem \ref{pseudoconverse} and Theorem \ref{pseudo1conv}.
\end{remark}
	 
	 \section{Illustration via Differential Equations}\label{sec4}
	In this section, we present concrete situations where our results find application. The applications of the results developed are illustrated through three types of differential equations:
	 \begin{enumerate}
	 	\item $g_1y'+ g_2y= g_3$ with $g_1$ and $g_2$ are periodic functions with same period.
	 	\item $y^{''}+ g_1y^{'}+ g_2y= g_3$ with $g_1$ is constant and $g_2$ is periodic.
	 	\item $y^{''}+ g_1y^{'}+ g_2y= g_3$ with $g_1, g_2$ are periodic functions with the same period. 
	 \end{enumerate}
	 \subsection{$g_1y'+ g_2y= g_3$ with $g_1$ and $g_2$ are periodic functions with same period.} Linear differential equations of the form $g_1y' + g_2y = g_3$ arise in various concrete applications such as complex dynamical systems and fluid mechanics. The operator theoretic approach to solving such equations or obtaining approximate solutions requires knowledge of the spectral behavior of the corresponding linear operators. In non-normal cases, understanding the behavior of pseudospectrum becomes crucial. Here, we observe that the path-connectedness of pseudospectrum is related to the oscillations of the potential.
	 
	 The finite difference approximation to the equation results in the following. Let $x_0$ be in the domain of $y$, then
	 \[
	 g_1(x_0)\left(\frac{y(x_0+h)-y(x_0-h)}{2h} \right)+ g_2(x_0)y(x_0)= g_3(x_0),
	 \]
	 and
	 \[
	 -g_1(x_0) y(x_0-h)+ 2h g_2(x_0)y(x_0)+ g_1(x_0)y(x_0+h)= 2h g_3(x_0).
	 \]
	 Assume without loss of generality that period of $g_1$ and $g_2$ is $1$. Choose $p$ such that $ph=1$. Denote $a_{i}= 2h\,g_2(x_0+ih)$ and $b_i= g_1(x_0+ih)$ for $i= 0, 1, \ldots, p-1$. This results in a periodic Jacobi operator defined by
	 \begin{align}\label{eqn1}
	 	J= \begin{bmatrix}
	 		\\
	 		\ddots&\ddots&\ddots\\
	 		&-b_{p-1}&a_0&b_0\\
	 		&&-b_0&a_1&b_1\\
	 		&&&\ddots&\ddots&\ddots\\
	 		&&&&-b_{p-2}&a_{p-1}&b_{p-1}\\
	 		&&&&&-b_{p-1}&a_0&b_0\\
	 		&&&&&&\ddots&\ddots&\ddots\\
	 		\\
	 	\end{bmatrix}
	 \end{align}
	 The periodic Jacobi operator results a doubly infinite block Toeplitz operator with matrix valued symbol 
	 \begin{align}\label{eqn2}
	 	\phi(\theta)= \begin{bmatrix}a_0&b_0&&&-b_{p-1}e^{-i\theta}\\-b_0&a_1&b_1&&&\\
	 		&\ddots&\ddots&\ddots&&\\&&-b_{p-3}&a_{p-2}&b_{p-2}\\
	 		b_{p-1}e^{i\theta}&&&-b_{p-2}&a_{p-1}\end{bmatrix}, \ \ -\pi\leq \theta\leq \pi.
	 \end{align}
	 \begin{remark}\label{pseudo3}
	 	Let $J$ be the discrete differential operator defined in \textnormal{\eqref{eqn1}} (the doubly infinite block Toeplitz operator with symbol $\phi$ defined in \textnormal{\eqref{eqn2}}). Then $a_i= 2hg_2(x_0+ih), b_i= g_1(x_0+ih)$, and $c_i= -b_{i+1}= -g_1(x_0+(i+1)h)$ for $i= 0, 1, \ldots, p-1$. Thus
	 	\begin{align*}
	 		\omega_{a_0}&= 2h\max\{|g_2(x_0+ih)-g_2(x_0)|: i= 0, 1,\ldots, p-1\},\\
	 		\omega_{b_0}&= \max\{|g_1(x_0+ih)-g_1(x_0)|: i= 0, 1,\ldots, p-1\},\\
	 		\omega_{b_0c}&= \max\{|g_1(x_0+ih)+g_1(x_0)|: i= 0, 1,\ldots, p-1\},
	 	\end{align*}
	 	and
	 	\begin{align*}
	 		\omega_{bc}&= \max\{|g_1(x_0+ih)+ g_1(x_0+(i+1)h)|: i= 0, 1,\ldots, p-1\}.
	 	\end{align*}
	 	Hence, if $\displaystyle \epsilon\geq 3\left(\sqrt{(p-1)\left(\omega_{a_0}^2+\omega_{b_0}^2+\omega_{b_0c}^2\right)}\right)$ then $\Lambda_{\epsilon}(A)$ is path-connected (from Theorem \ref{pseudo2}). Conversely, if $\Lambda_{\epsilon}(A)$ is path-connected for some $\e\geq \omega_{bc}$, then $\displaystyle \omega_a\leq 2(\epsilon+\omega_{bc})(p-1)$ (from Theorem \ref{pseudo2conv}).
	 \end{remark}
	 \subsection{$y^{''}+ g_1y^{'}+ g_2y= g_3$ with $g_1$ is constant and $g_2$ is periodic.} The finite difference approximation to the equation results in the following. Let $x_0$ be in the domain of $y$, then
	 \[
	 \frac{y(x_0+h)+ y(x_0-h)- 2\,y(x_0)}{h^2}+ g_1\left(\frac{y(x_0+h)- y(x_0-h)}{2h}\right)+ g_2(x_0) y(x_0)= g_3(x_0),
	 \]
	 \[
	 \left(\frac{1}{h^2}- \frac{g_1}{2h}\right)y(x_0-h)+ \left(g_2(x_0)- \frac{2}{h^2}\right)y(x_0)+ \left(\frac{1}{h^2}+ \frac{g_1}{2h}\right)y(x_0+h)= g_3(x_0),
	 \]
	 and
	 \[
	 (2-hg_1)\ y(x_0-h)+ (2h^2g_2(x_0)-4)\ y(x_0)+ (2+hg_1)\ y(x_0+h)= 2h^2g_3(x_0).
	 \]
	 Assume without loss of generality that the period of $g_2$ is 1. Choose $p$ such that $ph= 1$. Denote $a_i= 2h^2g_2(x_0+ih)-4$ for $i= 0,1,\ldots p-1$, $c= 2-hg_1$, and  $b= 2+ hg_1$. This results in a periodic Jacobi operator defined by 
	 \begin{align}\label{eqn3}
	 	J= \begin{bmatrix}
	 		\ddots&\ddots\\
	 		\ddots&a_0&b&\\
	 		&c&a_1&b\\
	 		&&\ddots&\ddots&\ddots\\
	 		&&&c&a_{p-1}&b\\
	 		&&&&c&a_0&\ddots\\
	 		&&&&&\ddots&\ddots
	 	\end{bmatrix}.
	 \end{align}
	 The periodic Jacobi operator results a doubly infinite block Toeplitz operator 
	 with matrix valued symbol 
	 \begin{align}\label{eqn4}
	 	\phi(\theta)= \begin{bmatrix}a_0&b&&&ce^{-i\theta}\\c&a_1&b&&&\\
	 		&\ddots&\ddots&\ddots&&\\&&c&a_{p-2}&b\\
	 		be^{i\theta}&&&c&a_{p-1}\end{bmatrix}, \ \ -\pi\leq \theta\leq \pi.
	 \end{align}
	 \begin{remark}
	 	Let $J$ be the discrete differential operator defined in \eqref{eqn3} (i.e., the doubly infinite block Toeplitz operator with symbol $\phi$ defined in \eqref{eqn4}. Then $|b-c|= 2h|g_1|$, 
	 	\begin{align*}
	 		\omega_{a_0}&= 2h^2\max\{|g_2(x_0+ih)-g_2(x_0)|: i= 1,\ldots, p-1\},
	 	\end{align*}
	 	and
	 	\begin{align*}
	 		\omega_{a}&= 2h^2\max\{|g_2(x_0+ih)-g_2(x_0+jh)|: i, j= 0, 1,\ldots, p-1\}.
	 	\end{align*}
	 	If $\displaystyle \epsilon\geq 3\omega_{a_0}$ then $\Lambda_{\epsilon}(A)$ is path-connected (from Theorem \ref{pseudo}). Conversely if $\Lambda_{\epsilon}(A)$ is path-connected for some $\e\geq 2h|g_1|$, then $\displaystyle \omega_a\leq 2(\epsilon+|b-c|)(p-1)$ (from Theorem \ref{pseudoconverse}).
	 \end{remark}
	 
	 \subsection{$y^{''}+ g_1y^{'}+ g_2y= g_3$ with $g_1, g_2$ are periodic functions with the same period}
	 Let $x_0$ be in the domain of the solution $y$. The finite difference approximation to the equation results the periodic Jacobi operator defined by
	 \begin{align}\label{eqn5}
	 	J= \begin{bmatrix}
	 		\ddots&\ddots\\
	 		\ddots&a_0&b_0&\\
	 		&c_0&a_1&b_1\\
	 		&&\ddots&\ddots&\ddots\\
	 		&&&c_{p-2}&a_{p-1}&b_{p-1}\\
	 		&&&&c_{p-1}&a_0&\ddots\\
	 		&&&&&\ddots&\ddots
	 	\end{bmatrix},
	 \end{align}
	 where $a_i= h^2g_2(x_0+ih)-4$, $b_i= 2+ h g_1(x_0+ih)$, and $c_i= 2- hg_1(x_0+(i+1)h)$, $i= 0, 1, \ldots, p-1$. The periodic Jacobi operator results a doubly infinite block Toeplitz operator with matrix valued symbol 
	 \begin{align}\label{eqn6}
	 	\phi(\theta)= \begin{bmatrix}a_0&b_0&&&c_{p-1}e^{-i\theta}\\c_0&a_1&b_1&&&\\
	 		&\ddots&\ddots&\ddots&&\\&&c_{p-3}&a_{p-2}&b_{p-2}\\
	 		b_{p-1}e^{i\theta}&&&c_{p-2}&a_{p-1}\end{bmatrix}, \ \ -\pi\leq \theta\leq \pi.
	 \end{align}
	 \begin{remark}
	 	Let $J$ be the discrete differential operator defined in \eqref{eqn5} (i.e., the doubly infinite block Toeplitz operator with symbol $\phi$ defined in \eqref{eqn6}. Then 
	 	\begin{align*}
	 		\omega_{a_0}&= h^2\max\{|g_2(x_0+ih)-g_2(x_0)|: i= 0, 1,\ldots, p-1\},\\
	 		\omega_{b_0}&= h\max\{|g_1(x_0+ih)-g_1(x_0)|: i= 0, 1,\ldots, p-1\},\\
	 		\omega_{b_0c}&= h\max\{|g_1(x_0+ih)+g_1(x_0)|: i= 0, 1,\ldots, p-1\},
	 	\end{align*}
	 	and
	 	\begin{align*}
	 		\omega_{bc}&= h\max\{|g_1(x_0+ih)+ g_1(x_0+(i+1)h)|: i= 0, 1,\ldots, p-1\}.
	 	\end{align*}
	 	If $\displaystyle \epsilon\geq 3\left(\sqrt{(p-1)\left(\omega_{a_0}^2+\omega_{b_0}^2+\omega_{b_0c}^2\right)}\right)$ then $\Lambda_{\epsilon}(A)$ is path-connected (from Theorem \ref{pseudo2}). Conversely if $\Lambda_{\epsilon}(A)$ is path-connected for some $\e\geq \omega_{bc}$, then $\displaystyle \omega_a\leq 2(\epsilon+\omega_{bc})(p-1)$ (from Theorem \ref{pseudo2conv}).
	 \end{remark}
	 
	 \subsection{Numerical Illustrations}
	 Consider a non-normal periodic Jacobi operator $J$ with matrix valued symbol defined by
	 \[
	 \phi(\theta)= \begin{bmatrix}
	 	-1&1&0&0&-e^{i\theta}\\
	 	-1.25&-0.5&1.25&0&0\\
	 	0&-1.5&0&1.5&0\\
	 	0&0&-1.75&0.5&1.75\\
	 	2e^{i\theta}&0&0&-2&1
	 \end{bmatrix}, \ \ \ -\pi\leq \theta\leq \pi.
	 \]
	 Write $\phi(\theta)= \phi_1(\theta)+ \phi_2(\theta)$, where 
	 \[
	 \phi_1(\theta)= \begin{bmatrix}
	 	-1&1&0&0&e^{-i\theta}\\
	 	1&-1&1&0&0\\
	 	0&1&-1&1&0\\
	 	0&0&1&-1&1\\
	 	e^{i\theta}&0&0&1&-1
	 \end{bmatrix}
	 \ \ \ \text{and} \ \ \ \phi_2(\theta)= \begin{bmatrix}
	 	0&0&0&0&-2e^{-i\theta}\\
	 	-2.25&0.5&0.25&0&0\\
	 	0&-2.5&1&0.5&0\\
	 	0&0&-2.75&1.5&0.75\\
	 	e^{i\theta}&0&0&-3&2
	 \end{bmatrix}.
	 \]
	 Then $\phi_1(\theta)$ is normal for every $-\pi\leq \theta\leq \pi$. The following figures are obtained, and it turns out that  $\displaystyle \sigma(J)= \bigcup_{-\pi \leq\theta\leq \pi} \sigma(\phi(\theta))$ is not path-connected, but $\displaystyle \bigcup_{-\pi\leq \theta\leq \pi} \Lambda_{\|\phi_2(\theta)\|}(\phi(\theta))$ is path-connected.
	 \begin{center}
	 	\begin{tabular}{cc}
	 		\includegraphics[width=.4\linewidth]{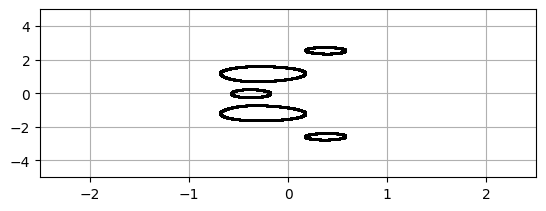}&
	 		\includegraphics[width=.4\linewidth]{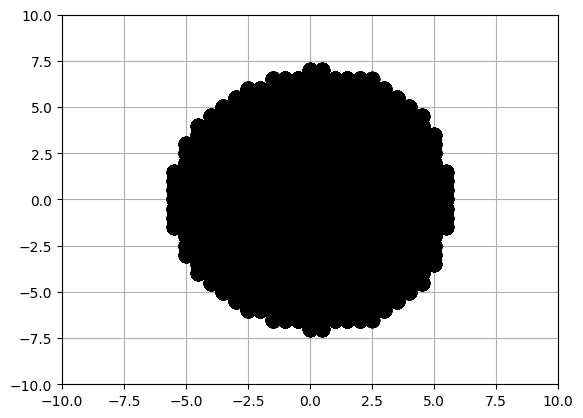}\\
	 		$\displaystyle \sigma(J)= \bigcup_{-\pi \leq\theta\leq \pi} \sigma(\phi(\theta))$&
	 		$\displaystyle \bigcup_{-\pi\leq \theta\leq \pi} \Lambda_{\|\phi_2(\theta)\|}(\phi(\theta))$
	 	\end{tabular}
	 \end{center}
	 
	\section{Concluding Remarks and Future Problems}
We conclude this article by highlighting some intriguing avenues for future research. Understanding the connection between the number of spectral gaps and the properties of the potential function has been a focal point in the continuous case \cite{Fla, goldberg-1, goldberg-2, hoschd-1}. Exploring discrete analogues of such results poses an interesting and general question. Additionally, one can anticipate discovering a relationship between the number of spectral gaps and the periodicity of the potential sequence.

To provide further clarity on this matter, let's consider an example.
	\subsection{Number of gaps and the periodicity index}
In the context of self-adjoint operators, it is imperative to conduct a detailed investigation into the relationship between the number of spectral gaps and the period. A notable example from \cite{kiran} provides some insights into this connection, revealing that the number of spectral gaps in this case is precisely $p-1$.

Consider the symbol $f_0(\theta)$ defined as follows:

	\[
	f_0(\theta)=\left[ {\begin{array}{*{20}c}
			{0} & 1 & {} & {} & {} & {e^{-i\theta}}  \\
			1 & {0 } & 1 & {} & {} & {}  \\
			{} & 1 & . & . & {} & {}  \\
			{} & {} & . & . & . & {}  \\
			{} & {} & {} & . & 0 &  1 \\
			{e^{i\theta}} & {} & {} & {} & 1 & {1}  \\
	\end{array}} \right].
	\]

In this specific example, the number of spectral gaps is established to be $p-1$. Generally, one can expect that if $A$ is a self-adjoint periodic Jacobi operator with an essential period of $p$, the number of gaps in the essential spectrum of $A$ will be on the order of $p$.
\subsection{The Optimal Inequalities}
The inequalities presented in Theorem \ref{pseudoconverse}, Theorem \ref{pseudo1conv}, and Theorem \ref{pseudo2conv} may not offer the most optimal bounds. Delving into the exploration and derivation of improved inequalities stands out as a valuable quantitative problem with potential applications. Furthermore, understanding how these inequalities contribute to obtaining approximate solutions for the differential equations discussed in Section \ref{sec4} remains an open question. Investigating their practical implications in the context of solving or approximating solutions to specific differential equations could provide additional insights. The enduring impact of B\"org's theorem, echoing through decades and captivating the interests of pure and applied mathematicians as well as physicists, is expected to maintain its legacy status for many more years to come.
	
	\subsection*{Acknowledgements} 
	Krishna Kumar G.  (SQUID-1984-KG-1801) is supported by the SERB MATRICS grant with project reference no. MTR/2021/000028, and  KSCSTE, Government of Kerala, India through Kerala State Young Scientist Award Research Grant, Order No. 88/2023/KSCSTE dated 13-03-2023. V. B. Kiran Kumar is supported by KSCSTE, Kerala via KSYSA-Research Grant.
	%\nocite{*}
	\bibliographystyle{amsplain}
	\bibliography{References}
	
\end{document}